\documentclass[a4paper,12pt]{amsart}
\usepackage[utf8]{inputenc}

\usepackage{url}
\usepackage[margin=1in]{geometry}  
\usepackage{epsfig}
\usepackage{color}
\usepackage{graphicx}              
\usepackage{amsmath}
\usepackage{amscd}               
\usepackage{amsfonts}
\usepackage{amssymb}             
\usepackage{amsthm}                
\usepackage{epsfig}
\usepackage{mathrsfs}
\usepackage{hyperref}
\usepackage{enumerate}

\numberwithin{equation}{section}

\newtheorem{theorem}{Theorem}

\newtheorem*{theorem*}{Theorem}
\newtheorem{lemma}[equation]{Lemma}
\newtheorem*{lemma*}{Lemma}
\newtheorem{proposition}[equation]{Proposition}
\newtheorem*{proposition*}{Proposition}
\theoremstyle{definition}

\theoremstyle{remark}

\newcommand{\Gal}{\operatorname{Gal}}
\newcommand{\Stab}{\operatorname{Stab}}

\newcommand{\SYM}{\operatorname{Sym}}
\newcommand{\Jacquet}{\operatorname{Jacquet}}
\newcommand{\Disc}{\operatorname{Disc}}

\newcommand{\vol}{\operatorname{vol}}

\newcommand{\one}{\mathbf{1}}

\newcommand{\bC}{\mathbb{C}}
\newcommand{\bH}{\mathbb{H}}

\newcommand{\bP}{\mathbb{P}}
\newcommand{\bQ}{\mathbb{Q}}
\newcommand{\bR}{\mathbb{R}}

\newcommand{\zed}{\mathbb{Z}}

\newcommand{\GO}{\mathrm{GO}}
\newcommand{\GL}{\mathrm{GL}}

\newcommand{\SL}{\mathrm{SL}}

\newcommand{\sF}{{\mathscr{F}}}

\newcommand{\sL}{{\mathscr{L}}}

\newcommand{\sS}{{\mathscr{S}}}

\title{The shape of quartic fields}
\author{Robert Hough}
\address[Robert Hough]{Department of Mathematics, Stony Brook University, 100 Nicolls Road, Stony Brook, NY 11794}
\email{robert.hough@stonybrook.edu}
\thanks{Robert Hough is supported by NSF Grant DMS-1712682, ``Probabilistic methods in discrete structures
and applications.''}
\thanks{The author thanks Peter Sarnak for a helpful conversation.}
\begin{document}

\begin{abstract}
We use the method of Shintani, as developed by Taniguchi and Thorne, to prove the joint cuspidal equidistribution  of the shape of quartic fields paired with the shape of its cubic resolvent, when the fields are ordered by discriminant.  Our estimate saves a small power in the corresponding Weyl sums.
\end{abstract}

\maketitle

\section{Introduction}
Let $R$ be a commutative ring with unit. An $n$-ic ring over $R$ is a free rank $n$ $R$ module with a ring structure.  We consider rings which are discrete subrings of full dimension in $\bR^{r_1} \times \bC^{r_2}$.  In this case the ring is a lattice and the ring multiplication is determined by the geometric shape of this lattice; the condition of integrality necessarily imposes constraints on those lattices which may occur.

Let $K/\bQ$ be a quartic field with $r_1$ real and $r_2$ complex embeddings, $r_1 + 2r_2 = 4$.  Let the real embeddings be $\sigma_1, ..., \sigma_{r_1}$, and the complex embeddings $\sigma_{r_1+1}, ..., \sigma_{r_1+r_2}$.  The map $\sigma: K \to \bR^{r_1}\times \bC^{r_2}$,
\begin{equation}
\sigma(x) = (\sigma_1(x), ..., \sigma_{r_1}(x), \sigma_{r_1+1}(x), ..., \sigma_{r_1+r_2}(x))
\end{equation}
is called the canonical embedding of $K$.  The ring of integers $Q$ of $K$ forms a lattice in $\bR^{r_1} \times \bC^{r_2} \cong \bR^n$ of covolume $2^{-r_2}|D|^{\frac{1}{2}}$, where $D$ is the field discriminant.  Since the discriminant tends to infinity in any enumeration of fields, and since 1 is always contained in the ring of integers as a short vector, define the \emph{shape} $\Lambda_4(Q)$ of $Q$ to be the lattice projected in the plane orthogonal to 1, and rescaled to have unit covolume.

In Bhargava's parametrization of quartic rings and fields \cite{B04} there is a unique cubic resolvent ring $R$ associated to the ring of integers of $Q$.  This cubic ring $C$ has the same discriminant as $Q$, which specifies whether $C \otimes \bR$ embeds in $\bR^3$ or $\bR \times \bC$, and thus has a lattice shape $\Lambda_3(C)$, which is the projection of $C$ in the plane orthogonal to 1 in this space rescaled to have covolume one.

In \cite{BH16} it is shown that when $S_n$ fields, $n = 3, 4, 5$ of a given signature are ordered by increasing discriminant, the shape of the ring of integers $\Lambda_n$ becomes asymptotically equidistributed in the space of lattices
\begin{equation}
\sS_{n-1} := \GL_{n-1}(\zed) \backslash \GL_{n-1}(\bR)/\GO_{n-1}(\bR)
\end{equation}
 with respect to the induced probability Haar measure. The identification of $\Lambda_n$ with a point in this space is made by choosing a base lattice. 
 
 A natural basis of functions in which to study equidistribution of lattices  consist in joint eigenfunctions of the Casimir operators and their $p$-adic analogues, the Hecke operators.  This splits the space into the orthogonal direct sum of the constant function, Eisenstein series and cusp forms.  This note gives a power saving estimate in the cuspidal spectrum for the joint shape of $Q$ and its cubic resolvent $C$ when quartic fields are ordered by discriminant.
 \begin{theorem}\label{main_theorem}
 	Let $\phi_3$ and $\phi_2$ be cuspidal Hecke-eigenforms on $\sS_3$ and $\sS_2$, respectively.  Let $\psi$ be a smooth function of compact support on $\bR^+$.  For each signature $r_2 =i$, $i = 0, 1, 2$, as $X \to \infty$, for any $\epsilon > 0$,
 	\begin{equation}
 	\sum_{\Gal(K/\bQ) = S_4, r_2 = i}\psi\left(\frac{|\Disc(K)|}{X}
\right)\phi_3(\Lambda_4(Q))\phi_2(\Lambda_3(C)) \ll_{\phi, \epsilon}
X^{\frac{23}{24}+\epsilon}.
 	\end{equation}
 	
 \end{theorem}

In \cite{B05} it is shown that for each signature, the number of $S_4$ quartic fields having discriminant of size at most $X$ is asymptotic to a constant times $X$.  Thus Theorem \ref{main_theorem} establishes a power savings in the joint equidistribution of the shape of the ring of integers $Q$ and the cubic resolvent $C$ in the cuspidal part of the spectrum.  As in our earlier work \cite{H17b} on the shape of cubic fields, it would be possible to treat $\Lambda_4$ and $\Lambda_3$ on the larger lattice space obtained by omitting the quotient on the right by the orthogonal group, thus obtaining further equidistribution.  We do not pursue this here.  However, whereas in our earlier work where the method extends to treat the Eisenstein spectrum, there are currently technical difficulties in doing so in the quartic case which we have not yet considered.  This is related to the difficulty in obtaining an asymptotic formula for the count of $S_4$ quartic fields using the zeta function approach of Yukie \cite{Y93},
 which we extend here by proving that a twisted version of Yukie's zeta function extends holomorphically to $\bC$, see Section \ref{proof_section} for a precise statement.

\subsection{History and discussion of method}
Sato and Shintani \cite{SS74} introduce $\zeta$ functions enumerating equivalence classes of integer points in prehomogeneous vector spaces.  As is familiar from the theory of $\zeta$ functions in multiplicative number theory, these generating functions may be completed at the real place to obtain objects enjoying a functional equation, although an Euler product representation is not typically available and the Riemann hypothesis is not usually satisfied. 

In the case of binary cubic forms, Shintani \cite{S72} makes a detailed study of the pole structure of the resulting object, obtaining refined information regarding the counting function of class numbers of binary cubic forms. This involves the integral forms in the singular set of the dual space.  To briefly motivate this, when $f = \one_{(\Disc x \in S)}$ is the characteristic function of a measurable set of discriminants in a prehomogeneous vector space $(G_\bR, V_\bR)$, the Fourier transform $\hat{f}$ evaluated at 0 gives the volume of the set of real forms with discriminant in  $S$.  Studying $\hat{f}(n)$ for integer vectors $n$ in the singular part of the dual space permits corrections to the naive estimate
\begin{equation}
|S \cap V_\zed| \approx \vol S,
\end{equation}
as in the Poisson summation formula.  The typical sets of interest are dilations, and dilating in time domain results in focusing in frequency domain, which in the case of prehomogeneous vector spaces restricts attention primarily to the singular set which is invariant under dilation.

Taniguchi and Thorne \cite{TT13a} study orbital zeta functions in the binary cubic form setting, which permit sieving for binary cubic forms with finitely many local conditions. As a result they obtain the best known asymptotic for the counting function of cubic fields in \cite{TT13b}, see also \cite{BST13} for an alternative approach to these estimates based upon the geometry of numbers.

In the case of pairs of ternary quadratic forms, Yukie \cite{Y93} studied the adelic $\zeta$-function, determining the location and order of its poles, but not their residues.  The adelic setting corresponds to enumeration of forms up to rational equivalence and parameterizes \'{e}tale quartic rings.  Bhargava \cite{B04} studied integral rather than rational equivalence and using this obtained asymptotic counts for the number of $S_4$ quartic fields using the geometry of numbers.  See \cite{BBP10} for the best power savings error term in this problem. 

In \cite{H17a} the author introduces a twisting automorphic cusp form into the original $\zeta$ function studied by Shintani, and in \cite{H17b} he refines this by combining with the local analysis of \cite{TT13a} to obtain quantitative spectral equidistribution estimates for the shape of cubic fields.  The Shintani $\sL$ functions which he introduces do not appear to satisfy a functional equation, but still have holomorphic continuation to $\bC$.  A weaker `split' functional equation familiar from earlier work in the theory of prehomogeneous vector spaces and similar to the proof of the functional equation in Tate's thesis is still satisfied, and this suffices as a substitute for most analytic purposes.

Here we return to introduce a twisting automorphic form to the zeta function of pairs of ternary quadratic forms as studied by Yukie, but we now include Bhargava's integral analysis into this theory.  One advantage of this approach is that we are able to study quantitatively the joint distribution of the shape of the ring of integers in a quartic field together with the shape of its cubic resolvent ring, refining earlier work of Bhargava and Harron \cite{BH16}.
\section{The space of pairs of ternary quadratic forms}
Given a ring $R$, 
let $(A, B) \in V_R=\SYM^2 R^3 \otimes R^2$ be a pair of ternary quadratic forms over $R$. Ternary quadratic forms are indicated in coordinates by writing
\begin{equation}
2 \cdot A = \begin{pmatrix} 2a_{11} & a_{12} & a_{13}\\ a_{12} & 2a_{22} & a_{23}\\ a_{13} & a_{23} & 2a_{33}\end{pmatrix}.
\end{equation}  
The group $ \GL_3(R) \times \GL_2(R)$ acts on $\SYM^2 R^3 \otimes R^2$.  Write $g_2 = \begin{pmatrix} r &s\\ t&u \end{pmatrix} \in \GL_2(R)$ and let $g_3 \in \GL_3(R)$.  The pair $(g_3, g_2)$ acts on $(A,B)$ via
\begin{equation}
(g_3, g_2) \cdot (A, B) = \left(r  g_3 A g_3^t + s g_3 B g_3^t, t g_3 A g_3^t + u g_3 B g_3^t\right).
\end{equation}
Set 
\begin{equation}
G_R = \GL_3(R) \times \GL_2(R)/K_R, \qquad K_R = \ker(\GL_3(R)\times \GL_2(R) \to \GL(V_R)  ).
\end{equation}

When $G_\bR$ acts on $V_\bR$ there is a single polynomial invariant of degree 12, called the discriminant $D$.  To obtain this, let $f(x,y)$ be the binary cubic form
\begin{equation}
f(x,y) = 4 \det(Ax - By).
\end{equation}
The discriminant of $(A, B)$ is the discriminant of $f(x,y)$.  We have
\begin{equation}
D((g_3, g_2) \cdot (A,B)) = \chi(g_2)D(A,B), \qquad \chi(g_2) = (\det g_2)^6.
\end{equation}
Bhargava \cite{B04} proves the following parameterization of quartic rings over $\zed$.
\begin{theorem}[\cite{B04}, Theorem 1]
	There is a canonical bijection between the set of $\GL_3(\zed) \times \GL_2(\zed)$-orbits on the space $\left(\SYM^2 \zed^3 \otimes \zed^2\right)^*$ of pairs of integral ternary quadratic forms and the set of isomorphism classes of pairs $(Q, R)$, where $Q$ is a quartic ring over $\zed$ and $R$ is a cubic resolvent ring of $Q$. Moreover, the discriminants of $Q$ and $R$ are equal to the discriminant of the pair $(A,B)$.
\end{theorem}
The map to the cubic ring is obtained as follows.  Given a pair $(A,B)$ of ternary quadratic forms, $f(x,y) = 4 \det(Ax + By)$ is a binary cubic form and $R$ is the cubic ring associated to $g$ under the Delone-Fadeev correspondence.  In particular, this ring is invariant under the $\GL_3$ action.  Similarly, the multiplication table of $Q$ is determined by determinants of $2\times 2$ slices of $\SYM^2 \zed^3 \otimes \zed^2$ and is thus invariant under the $\GL_2$ action, see \cite{B04} pp. 1340--1342.  The action on the pairs $(Q,R)$ is the natural one, that is, if $\langle\alpha_1, \alpha_2, \alpha_3\rangle$ is a basis for $Q/\zed$ and $\langle \beta_1, \beta_2 \rangle$ is a basis for $R/\zed$, then $\GL_3(\zed)$ acts by changing basis in $Q/\zed$ and $\GL_2(\zed)$ acts by changing basis in $R/\zed$.  The action extends to an action by the monoid $M_3(\zed)\times M_2(\zed)$ where $M_n(\zed)$ consists of $n\times n$ integer matrices of non-zero determinant.  In this case, when $M_3$ acts each basis vector of 
$Q/\zed$ is scaled by a further factor of the determinant and each basis vector of $R/\zed$ is scaled by the determinant squared, while when $M_2$ acts, all basis vectors are scaled by a factor of the determinant.  This condition is sufficient to guarantee that a ring structure is preserved, and is consistent with the action of $G_\bR$ on $V_\bR$.  Note that the determinant factor does not alter the \emph{shape} of the ring.

The content of a quartic ring $Q$ is the least $n \geq 1$ such that there is a quartic ring $Q'$ with $Q = \zed + n Q'$.  A quartic ring of content 1 is called primitive.  Let $\Delta^n$ be the set of order $n$ right Hecke points in $\SL_2(\zed)$,
\begin{equation}
\Delta^n = \left\{ \begin{pmatrix} a & b\\ 0 & d \end{pmatrix}: ad = n, 0 \leq b < d \right\}.
\end{equation}
In the proof of Corollary 4 of \cite{B04} it is shown that a primitive quartic ring has a unique cubic resolvent $R$, and if $Q' = \zed + nQ$ then the set of cubic resolvents of $Q'$ is given by
\begin{equation}
\{\gamma\cdot R: \gamma \in \Delta^n\}.
\end{equation}
\subsection{Prehomogeneous vector space structure}
Over $\bR$ the space  $V_\bR$ of pairs of real ternary quadratic forms has three open orbits $V_\bR^{(0)}, V_{\bR}^{(1)}, V_{\bR}^{(2)}$, where $V_{\bR}^{(i)}$ consists of those points in $\bP^2(\bR)$ in which the pair $(A, B)$ have $4-2i$ common zeros. Those points in $V_{\bR}^{(0)}, V_{\bR}^{(2)}$ have positive discriminant while points in $V_{\bR}^{(1)}$ have negative discriminant. Set $\varepsilon_j = (-1)^j$ for the sign of the discriminant on $V_{\bR}^{(j)}$.
We choose base points of discriminant $\pm 1$ in these spaces,
\begin{align}
x_0 &= \frac{1}{2^{\frac{5}{6}}}\left(\begin{pmatrix}-1\\&-1\\&&1 \end{pmatrix}, \begin{pmatrix}&1\\1\\&& \end{pmatrix} \right)\\ \notag
x_1 &= \frac{1}{2^{\frac{5}{6}}}\left(\begin{pmatrix}-1\\&1\\&&1 \end{pmatrix}, \begin{pmatrix}&1\\1\\&& \end{pmatrix} \right)\\ \notag
x_2 &= \frac{1}{2^{\frac{5}{6}}}\left(\begin{pmatrix}1\\&1\\&&1 \end{pmatrix}, \begin{pmatrix}&1\\1\\&& \end{pmatrix} \right).
\end{align}
A point $v \in V_{\bR}^{(i)}$ has stabilizer in $G_{\bR}$ of order $n_i = 24, 4, 8$ for $i = 0, 1, 2$. The complement of these spaces is the singular set $S_\bR$ where the discriminant vanishes.

Following \cite{Y93}, define a bilinear form on the space $\SYM^2(\bR^3)$ by
\begin{equation}
[x, y] = \sum_{i} x_{i,i}y_{4-i, 4-i} + \frac{1}{2} \sum_{1 \leq i < j \leq 3} x_{i,j}y_{4-j, 4-i}.
\end{equation}
Extend this to a bilinear form $[,]_V$ on $\SYM^2(\bR^3)\otimes \bR^2$ by
\begin{equation}
[(x_1,x_2), (y_1, y_2)]_V = [x_1, y_2]+[x_2,y_1].
\end{equation}
This identifies $V_\bR$ with its dual space.   Let $\tau_G$ be the long element of the Weyl group. There is an involution on $G$ defined by $g^\iota = \tau_G (g^{-1})^t \tau_G$. The bilinear form satisfies
\begin{equation}
[x, y]_V = [gx, g^\iota y]_V.
\end{equation}
The Fourier transform on $V_{\bR}$ of the function $f((g_3,g_2) \cdot)$ is given by $(\det g_2)^{-6} \sF f ((g_3, g_2)^\iota \cdot)$.

The space of integral pairs of ternary quadratic forms $L= V_\zed$ is a lattice in $V_\bR$.  The dual lattice $\hat{L} = \hat{V}_\zed$ consists of integral forms for which the mixed coefficients are even. Write $L_0$, (resp. $\hat{L}_0$) for those points in the  lattice (resp. dual lattice) with vanishing discriminant. For each $j=0,1,2$ and each $m>0$ the number of equivalence classes in $V_\zed^{(j)}$ of discriminant $\varepsilon_j m$ is finite.  Write $h_j(m)$ for this class number and choose representative forms $\{x_{i,m}^j\}_{i=1}^{h_j(m)}$. Also, choose group elements $g = (g_3, g_2)$ such that $x_{i,m}^j = g_{i,m}^j \cdot x_j$. Write $\Gamma_j(i,m)$ for the stabilizer in $G_\zed$ of $x_{i,m}^j$.

\subsection{Maximality and irreducibility}
A quartic ring of non-zero discriminant is said to be \emph{maximal} if it is not properly contained in another quartic ring as a subring.  A quartic ring $Q$ over $\zed$ is maximal if and only if $Q \otimes \zed_p$ is maximal for each prime $p$.

The following Lemma characterizing non-maximal rings is proved in \cite{B04}.
\begin{lemma}
	If $Q$ is any quartic ring that is not maximal at $p$, then there exists a $\zed$-basis $1, \alpha_1, \alpha_2, \alpha_3$ of $Q$ such that at least one of the following is true:
	\begin{enumerate}
		\item $\zed + \zed \cdot(\alpha_1/p) + \zed \cdot \alpha_2 + \zed \cdot \alpha_3$ forms a ring
		\item $\zed + \zed \cdot(\alpha_1/p) + \zed \cdot (\alpha_2/p) + \zed \cdot \alpha_3$ forms a ring
		\item $\zed + \zed \cdot (\alpha_1/p) + \zed \cdot(\alpha_2/p) + \zed \cdot (\alpha_3/p)$ forms a ring.
	\end{enumerate}
\end{lemma}
It is also shown (see pp. 1357--1359) that the conditions of non-maximality correspond to the following congruence conditions which may be checked modulo $p^2$.  Define
\begin{equation}
\lambda_{kl}^{ij}(A, B) = \det \begin{pmatrix} a_{ij} & b_{ij}\\ a_{kl} & b_{kl}\end{pmatrix}.
\end{equation}
The conditions are
\begin{enumerate}
	\item $\lambda_{22}^{11}, \lambda_{23}^{11}, \lambda_{33}^{11}, \lambda_{13}^{12}$ are multiples of $p$ and $\lambda_{12}^{11}, \lambda_{13}^{11}$ are multiples of $p^2$
	\item $\lambda_{13}^{11}, \lambda_{23}^{11}, \lambda_{13}^{12}, \lambda_{23}^{12}, \lambda_{22}^{13}, \lambda_{23}^{22}$ are all multiples of $p$, and $\lambda_{12}^{11}, \lambda_{22}^{11}, \lambda_{22}^{12}$ are multiples of $p^2$
	\item All the $\lambda_{kl}^{ij}$ are multiples of $p$.
\end{enumerate}
Let $\Phi_p(A, B)$ be the indicator that $(A, B)$ are non-maximal modulo $p$, and extend $\Phi_p$ to $\Phi_q$, $q$ square-free multiplicatively. Define the Fourier transform of $\Phi_p$ to be
\begin{equation}
\hat{\Phi}_q((A,B)) = \frac{1}{q^{24}}\sum_{(X,Y)\in (\zed/q^2 \zed)^{12}}\Phi_q((X,Y))e\left(\frac{[(X,Y), (A,B)]_V}{q^2}\right).
\end{equation}
A detailed study of these Fourier coefficients is being made by other authors.  We use only the trivial bound $\left|\hat{\Phi}_q((A,B))\right| \leq 1$.

Following \cite{B05}, a pair of integral ternary quadratic forms $(A, B)$ is
\emph{reducible} if $A$ and $B$ have a common zero in $\bP^2(\bQ)$ or if the
binary cubic form $f(x,y) = \det(Ax + By)$ is reducible.  Forms which are not reducible are called \emph{totally irreducible}.  Totally irreducible forms correspond to orders in either $A_4$ or $S_4$ quartic fields, and maximal totally irreducible forms correspond to the ring of integers in $A_4$ and $S_4$ fields.  The stabilizer in $G_\zed$ of a totally irreducible form is trivial.  

\section{Automorphic forms on $\GL_3(\bR)$}
We follow the convention of Goldfeld \cite{G06}. The upper half plane $\bH^3 \cong \GL_3(\bR)/(O_3(\bR)\cdot \bR^\times)$ consists of points
\begin{equation}
z_3 = x_3y_3 = \begin{pmatrix} 1 & x_{12} & x_{13}\\ 0 & 1 & x_{23}\\ 0&0&1\end{pmatrix}\begin{pmatrix} y_1y_2 & 0 & 0\\ 0 & y_1 & 0\\ 0&0&1\end{pmatrix}
\end{equation}
with Haar measure
\begin{equation}
d^*z_3 = dx_{12}dx_{13}dx_{23} \frac{dy_1 dy_2}{(y_1y_2)^3}
\end{equation}
This is invariant under scaling by the determinant.
The Siegel set
\begin{equation}
\Sigma_{\frac{\sqrt{3}}{2}, \frac{1}{2}}^3 = \left\{|x_{ij}| \leq \frac{1}{2}, y_i > \frac{\sqrt{3}}{2}\right\}
\end{equation}
contains a fundamental domain for the action of $\SL_3(\zed)$ on $\bH^3$.

The center of the universal enveloping algebra on $\GL_3$ is generated by a pair of Casimir operators $\Delta_1, \Delta_2$, see \cite{G06} p. 153 for their explicit expression in the above coordinates. Given type $\nu = (\nu_1, \nu_2) \in \bC^2$ define
\begin{equation}
I_\nu(z) = y_1^{\nu_1+2\nu_2}y_2^{2\nu_1 + \nu_2}.
\end{equation}
This is an joint eigenvector of the Casimir operators. Define 
\begin{equation}
w_3 = \begin{pmatrix} &&1\\ &-1 &\\1 &&\end{pmatrix}, \qquad u = \begin{pmatrix} 1 & u_{12} & u_{13}\\ & 1 & u_{23}\\ && 1\end{pmatrix}.
\end{equation}
Given parameters $m = (m_1, m_2)$ define character
\begin{equation}
\psi_m(u) = e^{2\pi i (m_1 u_{23} + m_2 u_{12})}.
\end{equation}
The Jacquet Whittaker function is
\begin{equation}
W_{\Jacquet}(z,\nu,\psi_m) = \int_{-\infty}^\infty \int_{-\infty}^\infty \int_{-\infty}^\infty I_\nu(w_3 \cdot u \cdot z) \overline{\psi_m(u)} du_{12}du_{13}du_{23}.
\end{equation}
This satisfies
\begin{equation}
W_{\Jacquet}(uz, \nu, \psi_m) = \psi_m(u)W_{\Jacquet}(z, \nu, \psi_m). 
\end{equation}

The estimate which we require regarding the Whittaker function is as follows.
\begin{theorem}[\cite{B84}, Theorem 2.1]
	There exist constants $N_1$ and $N_2$, depending in a continuous fashion on $\nu_1, \nu_2$ such that, if $n_1 > N_1$, $n_2 > N_2$, then 
	\begin{equation}
	y_1^{n_1}y_2^{n_2} W_{\Jacquet}\left(\begin{pmatrix} y_1y_2\\ & y_1 \\ &&1 \end{pmatrix}, (\nu_1, \nu_2), \psi_{1,1} \right)
	\end{equation}
	is bounded on $\bH^3$.
\end{theorem}

A Maass form $\phi$ of type $(\nu_1, \nu_2)$ is a smooth function in $L^2(\SL_3(\zed)\backslash \bH^3)$ which satisfies for Casimir operator $D$, $D\phi(z) = \lambda_D(\nu)\phi(z)$, where $\lambda_D(\nu)$ is the eigenvalue of $D$ on $I_\nu(z)$, and also, for each $U$ among
\begin{equation}
\SL_3(\bR), \begin{pmatrix} 1 &*&*\\ 0&1&0\\0&0&1\end{pmatrix}, \begin{pmatrix} 1&0&*\\0&1&*\\0&0&1\end{pmatrix},
\end{equation}
for all $z \in \bH^3$,
\begin{equation}
\int_{\SL_3(\zed)\cap U \backslash U} \phi(uz) du = 0.
\end{equation}

Let $U_2(\zed) = \begin{pmatrix} 1&\zed\\ &1\end{pmatrix}$.
A Maass form $\phi$ has a Fourier expansion
\begin{align}
\phi(z) &= \sum_{\gamma \in U_2(\zed)\backslash \SL_2(\zed)} \sum_{m_1=1}^\infty \sum_{m_2 \neq 0} \frac{A(m_1, m_2)}{|m_1m_2|}\\ \notag
& \times W_{\Jacquet}\left(\begin{pmatrix} |m_1m_2| \\ & m_1\\ &&1 \end{pmatrix}\begin{pmatrix} \gamma \\ & 1\end{pmatrix} z, \nu, \psi_{1, \frac{m_2}{|m_2|}} \right),
\end{align}
which may be expressed
\begin{align}
\phi(z) &= \sum_{\tiny{\begin{pmatrix} a&b\\c&d\end{pmatrix}} \in U_2(\zed)\backslash \SL_2(\zed)} \sum_{m_1 =1}^\infty \sum_{m_2 \neq 0}\frac{A(m_1,m_2)}{|m_1m_2|}e^{2\pi i\left[m_1 (cx_{13} + dx_{23}) +m_2 \Re\left(\frac{az_2 + b}{cz_2 + d}\right)\right]}\\
\notag &\times W_{\Jacquet}\left(\begin{pmatrix}|m_1m_2|\\ & m_1 \\&&1\end{pmatrix}\begin{pmatrix} \frac{y_1y_2}{|cz_2 + d|}\\ & y_1 |cz_2+d|\\&&1
\end{pmatrix}, (\nu_1, \nu_2), \psi_{1,1} \right),
\end{align}
where $z_2 = \begin{pmatrix}y_2 & x_{12}\\ &1 \end{pmatrix}$.
The Fourier coefficients satisfy $\frac{A(m_1, m_2)}{|m_1m_2|} = O(1)$. 
\begin{lemma}
	As $z$ varies in the Siegel set $\Sigma_{\frac{\sqrt{3}}{2}, \frac{1}{2}}^3$, $\phi(z)$ satisfies the bound, for any $n_1, n_2 > 0$,
	\begin{equation}
	\phi(z) \ll y_1^{-n_1}y_2^{-n_2}.
	\end{equation}
\end{lemma}
\begin{proof}
	Choose $N_1 > n_1$, $N_2 > n_2$ satisfying $N_1 > 4N_2$ and $N_2 >1$ and sufficiently large so that the Jacquet Whittaker function satisfies
	\begin{equation}
	W_{\Jacquet}((y_1, y_2)) \ll y_1^{-N_1}y_2^{-N_2}.
	\end{equation}
	Thus
	\begin{equation}
	|\phi(z)| \ll  \sum_{\tiny{\begin{pmatrix} a&b\\c&d\end{pmatrix}} \in U_2(\zed)\backslash \SL_2(\zed)} \sum_{m_1 =1}^\infty \sum_{m_2 \neq 0} (m_1 y_1 |cz_2 + d|)^{-N_1}\left(\frac{|m_2|y_2}{|cz_2 + d|^2} \right)^{-N_2}.
	\end{equation}	
	Thus sums over $m_1$ and $m_2$ converge to a constant, leaving the estimate
	\begin{equation}
	|\phi(z)| \ll y_1^{-N_1}y_2^{-N_2} \sum_{c,d, \gcd(c,d)=1} \frac{1}{|cz_2 + d|^{-2N_2}}.
	\end{equation}
	By \cite{I02} Lemma 2.10, uniformly in $z_2$ and $X$,
	\begin{equation}
	\#\{c,d: \gcd(c,d)=1, |cz_2 + d|^2< y_2 X\} \ll (1+X).
	\end{equation}
	Since $N_2>1$ it follows that the sum over $c,d$ is bounded uniformly in $z_2$.
\end{proof}

For $n \geq 1$ the Hecke operator $T_n$ is defined by
\begin{equation}
T_n \phi(z) = \frac{1}{n} \sum_{\substack{abc=n\\ 0 \leq c_1, c_2 < c\\ 0 \leq b_1 < b}} \phi\left(\begin{pmatrix} a& b_1 & c_1\\ &b&c_2\\ && c\end{pmatrix} z\right).
\end{equation}
These operators are normal, and commute with each other and with the Casimir operators.  If $\phi$ is a Maass form which is a joint eigenfunction of the Hecke algebra normalized such that $A(1,1)=1$ then
\begin{equation}
T_n \phi = A(n,1)\phi.
\end{equation}
The Fourier coefficients satisfy the multiplicativity relationship
\begin{align}
A(m_1m_1', m_2m_2') &= A(m_1, m_2)A(m_1', m_2'), \qquad (m_1m_2, m_1'm_2')=1\\
\notag A(n,1)A(m_1,m_2) &= \sum_{\substack{d_0d_1d_2 = n \\ d_1|m_1\\ d_2|m_2}} A\left(\frac{m_1 d_0}{d_1}, \frac{m_2 d_1}{d_2} \right)\\
\notag A(1,n)A(m_1, m_2) &= \sum_{\substack{d_0d_1d_2 = n\\ d_1|m_1 \\ d_2|m_2}} A\left(\frac{m_1 d_2}{d_1}, \frac{m_2 d_0}{d_2} \right)\\
\notag A(m_1, 1)A(1, m_2)&= \sum_{d|(m_1, m_2)}A\left( \frac{m_1}{d}, \frac{m_2}{d}\right).
\end{align}

Let 
\begin{equation}
w_3 = \begin{pmatrix} &&1\\ &-1\\1\end{pmatrix} 
\end{equation}
be the long element of the Weyl group on $\GL_3$.  The dual Maass form of $\phi$ is 
\begin{equation}
\tilde{\phi}(g) = \phi\left(w_3 \left(g^{-1}\right)^t w_3\right).
\end{equation}
This is a form of type $(\nu_2, \nu_1)$ with $m_1, m_2$ Fourier coefficient $A(m_2,m_1)$.

\subsection{Automorphic forms on $\GL_2(\bR)$}
Following Goldfeld, identify the upper half plane $\bH^2 \cong \GL_2(\bR)/(O_2(\bR)\cdot \bR^\times)$ with $\left\{\begin{pmatrix}y&x\\&1 \end{pmatrix}\right\}$, with Haar measure $\frac{dx dy}{y^2}$. Note that the involution 
\begin{equation}
g^\iota = \begin{pmatrix} &-1\\1\end{pmatrix} (g^{-1})^t \begin{pmatrix} &1\\-1\end{pmatrix}
\end{equation}
is the identity on $\bH^2$.
The Siegel set 
\begin{equation}
\Sigma_{\frac{\sqrt{3}}{2}, \frac{1}{2}}^2 = \left\{|x| \leq \frac{1}{2}, y > \frac{\sqrt{3}}{2}\right\}
\end{equation}
contains a fundamental domain for the action of $\SL_2(\zed)$ on $\bH^2$.

A weight 0 Maass cusp form of Laplace eigenvalue $\nu = \frac{1}{2} + it$ on $\SL_2(\zed)\backslash \bH^2$ has a Fourier development
\begin{equation}
f(z) = \sqrt{2\pi y}\sum_{n \neq 0} a(n)K_{\nu-\frac{1}{2}}(2\pi |n| y) e^{2\pi i nx}.
\end{equation}
If $f(z)$ is an eigenfunction of the $\GL_2$ Hecke algebra with $a(1) = 1$ then $T_n f = a(n)f$.  These eigenvalues satisfy 
\begin{equation}
a(m)a(n) = \sum_{d |(m,n)} a\left(\frac{mn}{d^2}\right).
\end{equation}
The best bound for $|a(m)|$ is $|a(m)| \ll |m|^{\frac{7}{64} +\epsilon}$, \cite{K03}.
As $y \to \infty$,
\begin{equation}
K_{\nu-\frac{1}{2}}(y) \sim \sqrt{\frac{\pi}{2y}} e^{-y}
\end{equation}
and in particular, $f(z)$ decays exponentially as $y \to \infty$.
\subsection{Conjugation invariant kernels}
Let $G^1 = G_{\bR}/\bR^+$, identified with
\begin{equation}
\{(g_3,g_2) \in \GL_3(\bR) \times \GL_2(\bR): \det g_3 = |\det g_2| = 1\}.
\end{equation}
Let $\psi = \psi_3 \otimes \psi_2$ be a smooth function of compact support on $G^1$ which is conjugation invariant, in the sense that $\psi(h^{-1}gh)=\psi(g)$, and supported on the conjugacy class of the diagonal matrices.  As a convolution kernel, $\psi$ commutes with translation since for $f \in C(G^1)$,
\begin{align}
\psi * L_h f(g) &= \int \psi(gx^{-1})f(hx)dx\\
\notag &= \int \psi(gx^{-1}h)f(x)dx = \int \psi(hgx^{-1})f(x)dx = \psi * f(hg).
\end{align} 
Thus convolution with $\psi$ commutes with the action of the Lie algebra.  

Let $\phi = \phi_3 \otimes \phi_2$ be the tensor of Hecke-eigen Maass cusp forms on $\GL_3(\bR)$ and $\GL_2(\bR)$.  Since convolution with $\psi$ commutes with the action of the Casimir operators and Hecke algebras on $\GL_3$ and $\GL_2$, it follows from strong multiplicity 1 (see e.g. \cite{G06} p.393) that $\psi$ acts by a scalar on $\phi$,
\begin{equation}
\psi * \phi = \Lambda_{\psi, \phi}\phi.
\end{equation}
\section{Orbital integral representation}
Let $f_G$ be a smooth, conjugation invariant function of compact support on $G^1$, as above.  Let $f_D$ be a smooth function of compact support on $\bR^+$.  Let $f_j$, $j = 0, 1, 2,$ be given by $f_j = f_G \otimes f_D$, interpreted as
\begin{equation}
f_j(x) = \sum_{\substack{g \in G^1, \lambda \in \bR^+\\ \lambda g \cdot x_j = x}}f_G(g)f_D(\lambda^{12}).
\end{equation}

Let $\phi = (\phi_3, \phi_2)$ be a pair of $\GL_3$ and $\GL_2$ Hecke-eigen Maass cusp forms of full level, normalized to have first Fourier coefficient equal to 1.  Let $q$ be a square-free integer.  Define the $q$-non-maximal orbital $L$-function in $\Re(s)>1$ by 
\begin{equation}
Z_q^j(f_j, \phi; s) = \int_{\lambda \in \bR^+} \lambda^{12s}  \int_{g \in G^1/G_{\zed}} \phi\left(g^{-1}\right)\sum_{x \in V(\zed)} \Phi_q(x) f_j\left(\lambda g x \right)dg\frac{d\lambda}{\lambda}.
\end{equation}
Also introduce the twisted Shintani $\sL$-functions, defined for $\Re(s)>1$ by
\begin{equation}
\sL_q^j(\phi, s) = \sum_{m=1}^\infty \frac{1}{m^s} \sum_{g_0 \cdot x_j = x_j}\sum_{i=1}^{h_j(m)}\Phi_q\left(x_{i,m}^j\right)\frac{\phi\left(g_{i,m}^jg_0\right)}{|\Gamma_j(i,m)|}.
\end{equation}

The twisted $\sL$-function may be recovered from the orbital $L$-function as follows.
\begin{lemma}
	For $\Re(s)>1$ we have the factorization
	\begin{equation}
	Z_q^j(f_j, \phi; s) = \frac{\Lambda_{f_G, \phi}}{12} \sL_q^j(\phi,s)\tilde{f}_D(s).
	\end{equation}
\end{lemma}
\begin{proof}
	Write
	\begin{align}
	&Z_q^j(f_j, \phi; s) = \int_{\lambda \in \bR^+} \lambda^{12s}  \int_{g \in G^1/G_{\zed}} \phi\left(g^{-1}\right)\sum_{x \in V(\zed)} \Phi_q(x) f_j\left(\lambda g x \right)dg\frac{d\lambda}{\lambda}\\
	\notag &= \int_{0}^\infty\lambda^{12s}\int_{G^1/G_{\zed}}\phi\left(g^{-1}\right)\sum_{m=1}^\infty \sum_{i=1}^{h_j( m)}\frac{\Phi_q(x_{i,m}^j)}{|\Gamma_j(i,m)|}\sum_{\gamma \in G_\zed}f_j(\lambda g\gamma g_{i,m}^j x_j)dg\frac{d\lambda}{\lambda}\\
	\notag &= \frac{1}{12} \sum_{m=1}^\infty \sum_{i=1}^{h_j( m)}\frac{\Phi_q(x_{i,m}^j)}{|\Gamma_j(i,m)|} \int_{G^1} \phi\left(g^{-1}\right)\sum_{g_0 x_j = x_j}f_G(gg_{i,m}^jg_0)dg \int_0^\infty \lambda^s f_D(\lambda m)\frac{d \lambda}{\lambda}\\
	&\notag = \frac{\Lambda_{f_G, \phi}}{12} \sL_q^j(\phi,s)\tilde{f}_D(s).
	\end{align}
\end{proof}
For $g = (g_3, g_2) \in G^1$, the Poisson summation formula gives
\begin{align}
\sum_{x \in V(\zed)}\Phi_q(x)f_j(g\cdot x) &= \sum_{a \in V(\zed/q^2 \zed)}\Phi_q(a)\sum_{x \in V(\zed)}f(g\cdot(q^2 x + a))\\
\notag &= \frac{1}{q^{24}}\sum_{a \in V(\zed/q^2\zed)}\Phi_q(a)\sum_{y \in \hat{V}(\zed)}e\left(\frac{1}{q^2}[y,a]_V \right)\hat{f}\left(g^\iota \frac{y}{q^2} \right)\\
\notag &= \sum_{y \in \hat{V}(\zed)}\hat{\Phi}_q(y)\hat{f}\left(g^\iota \frac{y}{q^2}\right).
\end{align}
Define truncated orbital functions
\begin{align}
Z_q^{j,+}(f_j, \phi;s) &=
\int_{\frac{1}{q^2}}^\infty
\lambda^{12 s}\int_{G^1/G_{\zed}}\phi\left(g^{-1}\right) \sum_{x \in V(\zed)} \Phi_q(x) f_j\left(\lambda g x \right)dg\frac{d\lambda}{\lambda}\\\notag
\hat{Z}_q^{j,+}(f_j, \tilde{\phi};1- s) &=
\int_{q^2}^\infty \lambda^{12-12 s}\int_{G^1/G_{\zed}}\tilde{\phi}\left(g^{-1}\right)
\sum_{x \in \hat{V}(\zed)} \hat{\Phi}_q(x)
\hat{f}_j\left( \lambda g \frac{x}{q^2}
\right)dg\frac{d\lambda}{\lambda}.
\end{align}
Splitting the integral over $\lambda$ at $\lambda = \frac{1}{q^2}$ and applying
Poisson summation obtains the split functional equation.
\begin{lemma}[Split functional equation]
	We have
	\begin{equation}
	Z_q^j(f_j, \phi; s) = Z_q^{j, +}(f_j, \phi; s) + \hat{Z}_q^{j,
+}\left(\hat{f}_j, \tilde{\phi}; 1-s\right) .
	\end{equation}
\end{lemma}

\section{Proof of Theorem}\label{proof_section}

As a first step, we verify
\begin{proposition}
 The twisted non-maximal function $\sL_q^j(\phi, s)$ extends to a holomorphic
function on $\bC$. In particular, for $s = \sigma + it$, in $\sigma \leq 1$ the orbital $L$-functions satisfy the bound
\begin{equation}
|Z_q^j(f_j, \phi; s)| \ll_{f_j, \sigma} q^{24(1-\sigma)}.
\end{equation} 
\end{proposition}

\begin{proof}
 It suffices to prove that the orbital $L$-function extends holomorphically. 
We apply the split functional equation and treat separately $Z_q^{j, +}(f_j,
\phi; s)$ and $\hat{Z}_q^{j,
+}\left(\hat{f}_j, \phi; 1-s\right).$
We use the estimate, for all $B>0$ and all $x \neq 0$,
\begin{equation}\left|f_j(x)\right|, \left|\hat{f}_j(x)\right| \ll_B
\|x\|_2^{-B},\end{equation} which follows since $\hat{f}_j$ is Schwarz class.

Let $\Sigma = \Sigma_{\frac{\sqrt{3}}{2}, \frac{1}{2}}^3 \times
\Sigma_{\frac{\sqrt{3}}{2}, \frac{1}{2}}^2$ be a Siegel set for
$G_\zed\backslash G_\bR/K$, $K = \GO_3(\bR)\times \GO_2(\bR)$. 
As $g = (x_3 y_3 k_3, x_2y_2 k_2)$ varies in $\Sigma$, \begin{equation}c(g) =
(y_3^{-1}x_3 y_3 k_3, y_2^{-1}x_2y_2 k_2)\end{equation} varies in a compact set,
and hence for $x \neq 0$, $\|c(g)^{-1}\cdot x\|_2 \asymp \|x\|_2$. Since $g^{-1}
\cdot x
= c(g)^{-1} \cdot ((y_3, y_2)^{-1}\cdot x)$, it suffices to consider only the
$y$ action.
 Write $y_3 = (y_{13}, y_{23})$.  We have $y_2, y_{13}, y_{23}$ are bounded
below away from 0 in $\Sigma$. Since $(y_3, y_2)$ acts on $x$ by dilations in
the coordinates by positive and negative fractional powers in $y_{13}, y_{23},
y_2$, it follows that for some $\delta > 0$, for $\|x\|_2 \geq 1$,
\begin{equation}
 \max\left(y_{13}, y_{23}, y_2, \left\|y^{-1} \cdot
x\right\|_2\right) \gg
\left\|x\right\|_2^\delta
\end{equation}
and hence by the rapid decay of $\phi$ in the cusps, for any $B>0$,
\begin{equation}
 \phi(g) f_j\left(g^{-1}\cdot x\right) \ll_B
\left\|x\right\|_2^{-B}, \qquad \tilde{\phi}(g) \hat{f}_j\left(g^{-1}\cdot x \right)
\ll_B
\left\|x\right\|_2^{-B}.
\end{equation} 

Note that $x = 0$ may be omitted from the sums defining $Z_q^{j,+}$ and
$\hat{Z}_q^{j,+}$.  In the first case this is because $f_j$ is supported away
from 0.  In the second case this is because the integral of $\phi$ vanishes.

Inserting our estimates with $B>12$, and using that $G^1/G_{\zed}$ has finite
volume and the bounds $\left|\Phi_q(x)\right| \leq 1$,
$\left|\hat{\Phi}_q(x)\right| \leq 1$, yields the bounds
\begin{align}
 \left|Z_q^{j,+}(f_j, \phi; s)\right| &\ll \int_{\frac{1}{q^2}}^{\infty}
\lambda^{12 \sigma} \sum_{0 \neq x \in V(\zed)} \min\left(1, \|\lambda
x\|_2^{-B}\right) \frac{d\lambda}{\lambda}\\
\notag & \ll \int_{\frac{1}{q^2}}^\infty \lambda^{12\sigma - 12}
\frac{d\lambda}{\lambda} \ll q^{24(1-\sigma)}
\end{align}
and
\begin{align}
 \left|\hat{Z}_q^{j,+}\left( \hat{f}, \phi; 1-s\right)\right| &\ll
\int_{q^2}^\infty \lambda^{12-12 \sigma} \sum_{0 \neq x \in V(\zed)}
\min\left(1, \left\|\frac{\lambda x}{q^2}\right\|_2^{-B}\right)
\frac{d\lambda}{\lambda}\\
\notag & \ll q^{24(1-\sigma)}.
\end{align}

\end{proof}

The second main step in this argument estimates the cusp form summed over
reducible maximal pairs $(A,B)$.  In the remainder of the paper, for $(A,B) \in V_\zed^{(j)}$, define
\begin{equation}
\phi((A,B)) = \sum_{g \in G_\bR: g \cdot x_j = (A,B)}\phi(g).
\end{equation}
\begin{proposition}\label{reducible_prop}
 We have the estimate, for each $j = 0, 1, 2$, 
 \begin{equation}
  \sum_{\substack{(A,B) \in V_\zed^{(j),*}/G_\zed \\ 0 < |\Disc(A,B)| < X}}
|\phi((A,B))| \ll X^{\frac{11}{12}+\epsilon},
 \end{equation}
with the $*$ restricting summation to pairs $(A, B)$ which are maximal and
reducible.
\end{proposition}

\begin{proof}
 This builds on the estimates of \cite{B05} starting from p. 1041.  A
point $(A,B)$ in the sum with discriminant between $\frac{X}{2}$ and $X$ has a
representative such that $X^{-\frac{1}{12}} (y_3, y_2)^{-1}(A,B)$ lies within a
compact neighborhood bounded away from 0 for some $(y_3, y_2)$ in the Siegel set
$\Sigma$. Lemmas 12 and 13 of \cite{B05} show that the number of such reducible
points with $a_{11} \neq 0$ is $O(X^{\frac{11}{12}+\epsilon})$.  Meanwhile, if
$a_{11} = 0$ then there are 11 free parameters $a_{ij}$, $b_{ij}$, and hence
the number of such points for fixed $(y_3, y_2)$ is bounded by a polynomial in
$y_3, y_2$ times $X^{\frac{11}{12}}$.  The claim now follows due to the
rapid decay of the cusp forms in the parameters $y_3, y_2$.

\end{proof}

The following is a variant of \cite{B05}, Proposition 23.
\begin{proposition}\label{sifting_prop}
 Let $q$ be square-free and let $N(W_q, X)$ be the number of classes of pairs
$(A,B)$ which are non-maximal at all primes $p|q$ and have $0 < |\Disc((A,B))|
\leq X$.  We have
\begin{equation}
 N(W_q, X) = O\left(\frac{X \log X e^{O(\omega(q))}}{q^2} \right).
\end{equation}

\end{proposition}

\begin{proof}
It follows from \cite{Y93} that the number of equivalence classes of integral
pairs $(A,B)$ with discriminant of size between 1 and $X$ is $O(X \log X)$. 

 Nakagawa \cite{N96} verifies that the number of index $k = \prod p_i^{e_i}$
subrings of a maximal quartic ring $Q$ is 
\begin{equation}
 O_\epsilon\left(\prod_i p_i^{(2+\epsilon)\lfloor e_i/4\rfloor} \right).
\end{equation}
Note that the argument of \cite{N96} treats $Q$ which is the ring of integers
of a quartic number field, but as the count may be checked locally, the bound
applies to maximal reducible $Q$ as well, i.e. because there are rings of
integers which satisfy any finite list of local constraints which are maximal.
The discriminant of an index $k$ subring of $Q$ is $k^2 \Disc(Q)$. 

The argument now proceeds essentially as in \cite{B05}. A quartic ring $Q$, of
content $n$, $Q = \zed + nQ'$ has $\sigma(n) = \sum_{d|n}d$ cubic resolvent
rings and discriminant $\Disc(Q) = n^6 \Disc(Q')$.
It follows that
\begin{align}
 N(W_q, X) &< \sum_{n=1}^\infty \frac{\sigma(n)}{n^6} \prod_{p|q} \left(\sum_{e
=1}^\infty \frac{p^{(2+\epsilon) \lfloor e/4\rfloor}}{p^{2e}} \right)
\prod_{p\nmid q}\left(\sum_{e=0}^\infty \frac{p^{(2+\epsilon) \lfloor
e/4\rfloor}}{p^{2e}} \right)O(X \log X) \\&\notag= O\left(\frac{X \log X
e^{O(\omega(q))}}{q^2} \right).
\end{align}

\end{proof}
\begin{proof}[Proof of Theorem \ref{main_theorem}]
Consider the sum
\begin{equation}
N_j(X; \psi, \phi) = \sum_{\substack{(A,B) \in V_\zed^{(j)}/G_\zed \\ \text{maximal}}} \psi\left(\frac{|\Disc(A,B)|}{X} \right)\frac{\phi((A,B))}{|\Stab_{G_\zed}((A,B))|}.
\end{equation}
By Proposition \ref{reducible_prop} (recall that the stabilizer of an irreducible element is trivial)
\begin{equation}
N_j(X; \psi, \phi) = \sum_{\substack{(A,B) \in V_\zed^{(j)}/G_\zed \\ \text{maximal, irreducible}}} \psi\left(\frac{|\Disc(A,B)|}{X} \right)\phi((A,B)) + O\left(X^{\frac{11}{12}+\epsilon} \right).
\end{equation}
The latter sum is over maximal orders in $A_4$ or $S_4$ quartic fields.  The number of $A_4$ fields with discriminant of size at most $X$ is $O\left(X^{\frac{7}{8}+\epsilon}\right)$ (\cite{B80}, \cite{W99}) so it suffices to prove the estimate for $N_j(X;\psi, \phi)$.

Let $\delta>0$ be a small parameter and let $Q = X^{\delta}$. By inclusion-exclusion,
\begin{align}
N_j(X;\psi, \phi) &= \sum_{q < Q} \mu(q) \sum_{\substack{(A,B) \in V_\zed^{(j)}/G_\zed}}\Phi_q((A,B)) \psi\left(\frac{|\Disc(A,B)|}{X} \right)\frac{\phi((A,B))}{|\Stab_{G_\zed}((A,B))|} \\ &\notag + \sum_{q \geq Q} \mu(q)\sum_{\substack{(A,B) \in V_\zed^{(j)}/G_\zed}}\Phi_q((A,B)) \psi\left(\frac{|\Disc(A,B)|}{X} \right)\frac{\phi((A,B))}{|\Stab_{G_\zed}((A,B))|}
\end{align}
The second sum is $O\left(\frac{X\log X}{Q^{1-\epsilon}}\right)$ by Proposition \ref{sifting_prop}.  By Mellin inversion, the first sum may be expressed
\begin{equation}
\sum_{q < Q} \mu(q) \oint_{\Re(s) = 2} \tilde{\psi}(s)X^s \sL_q^j(\phi,s)ds.
\end{equation}
Let $A>0$ be a parameter, and choose test function $f$ supported on $V_\zed^{(j)}$ such that $\Lambda_{f_G, \phi} \neq 0$ and such that $\tilde{f}_D(s) \neq 0$ in $\Re(s) > -A$ and satisfies $\frac{\tilde{\psi}(s)}{\tilde{f}_D(s)}\ll \frac{1}{|s|^2}$ in $-A < \sigma < 3$.  In the integral, replace
\begin{equation}
\sL_q^j(\phi,s) = \frac{12 Z_q^j(f_j,\phi;s)}{\Lambda_{f_G,\phi} \tilde{f}_D(s)}
.
\end{equation}

Shift the integral to $\Re(s) = -A$ and bound $|Z_q^{j}(f_j,\phi;s)| \ll_A
q^{24(1+A)}$, uniformly in $s$.  This gives a bound for the first sum of
$X^{-A}Q^{24A + 1}$.  Letting $A$ tend to $\infty$ we find we may take $\delta$ arbitrarily close to $\frac{1}{24}$ and obtain the estimate $O(X^{1-\frac{1}{24}+\epsilon})$.

\end{proof}
\bibliographystyle{plain}

\end{document}